\documentclass{amsart}
\usepackage{amsthm,amsmath,amsfonts,amssymb,enumerate,array,booktabs}%, showkeys}%,multirow,stmaryrd,esint}%,mathpazo}
\usepackage{hyperref}

\newtheorem{thm}{Theorem}[section]
\newtheorem{lemma}[thm]{Lemma}
\newtheorem{corol}[thm]{Corollary}
\newtheorem{prop}[thm]{Proposition}

\theoremstyle{definition}
\newtheorem{defi}[thm]{Definition}
\theoremstyle{remark}

\newtheorem{rem}[thm]{Remark}

\newcommand{\p}{\partial}
\newcommand{\vphi}{\varphi}

\DeclareMathOperator{\tr}{tr}

\DeclareMathOperator{\id}{id}
\DeclareMathOperator{\const}{const}
\DeclareMathOperator{\vol}{vol}

\DeclareMathOperator{\Ima}{Im}

\renewcommand{\phi}{\varphi}
\renewcommand{\epsilon}{\varepsilon}

\renewcommand{\bar}{\overline}
\renewcommand{\tilde}{\widetilde}
\DeclareMathOperator{\img}{Im}
\renewcommand{\Im}{\img}

\newcommand{\de}{\partial}
\newcommand{\debar}{{\bar \partial}}

\newcommand{\call}{\mathcal}
\newcommand{\M}{\mathcal M}
\renewcommand{\H}{\mathcal H}
\newcommand{\Svol}[2]{ \frac 1 {{#2}!} \eta_{#1} \wedge d\eta_{#1}^{#2}} %to be
\newcommand{\Cka}[1]{{C^{#1, \alpha}(g)}}
\allowdisplaybreaks
\numberwithin{equation}{section}
\title[On the geodesic problem for the Dirichlet metric]{On the geodesic problem for the Dirichlet metric
and the Ebin metric on the space of Sasakian metrics}
\author{Simone Calamai, David Petrecca and Kai Zheng}
\address{(S.~Calamai) Dip. di Matematica e Informatica ``U. Dini'' - Universit\`a di Firenze \endgraf Viale Morgagni 67A -  Firenze - Italy}
\email{simocala at gmail.com}
\address{(D.~Petrecca) Institut f\"ur Differentialgeometrie, Leibniz Universit\"at Hannover
 \endgraf  Welfengarten 1, 30167
- Hannover - Germany}
\email{petrecca at math.uni-hannover.de}
\address{(K.~Zheng) Mathematics Institute, University of Warwick
 \endgraf Coventry, CV4 7AL, UK}
\email{K.Zheng at warwick.ac.uk}
\thanks{The first and the second author are supported by INdAM}
\subjclass[2010]{53C55 (primary);  32Q15, 58D27 (secondary)}
\begin{document}

\maketitle

\begin{abstract}
We study the geodesic equation for the Dirichlet (gradient) metric in the space of K\"ahler potentials.
We first solve the initial value problem for the geodesic equation of the \emph{combination metric}, including the gradient metric. We then discuss a comparison theorem between it and the Calabi metric.
As geometric motivation of the combination metric, we find that the Ebin metric restricted to the space of type II deformations of a Sasakian structure
is the sum of the {Calabi metric} and the gradient metric.
\end{abstract}
\tableofcontents

\section*{Introduction}
This is the sequel of the previous paper \cite{calazheng} on the \emph{Dirichlet metric}, which here will be called gradient metric. We recall the background briefly.
The idea of defining a Riemannian structure on the space of all metrics on a
fixed manifold goes back to the sixties
with the work of Ebin \cite{ebin}.
His work concerns the pure Riemannian setting and, among other things, defines a
weak Riemannian metric on the space
$\M$ of all Riemannian metrics on a fixed compact Riemannian manifold $(M, g)$.
The geometry of the Hilbert manifold $\M$ was later studied by Freed and
Groisser in \cite{freedgroisser} and
Gil-Medrano and Michor in \cite{michor}. In particular the curvature and the
geodesics of $\M$ were computed.

Let $(M,\omega)$ be a compact K\"ahler manifold.
The space $\H$ of K\"ahler metrics cohomologous to $\omega$ is isomorphic to the space of the K\"ahler potentials modulo constants. It can be endowed with three different metrics, known as the \emph{Donaldson-Mabuchi-Semmes} $L^2$-metric \eqref{dms},
the \emph{Calabi} metric \eqref{calabi} and the \emph{Dirichlet (or gradient)} metric \eqref{gradient}.

The Calabi metric goes back to Calabi \cite{calabi1954space} and it was later studied by the first author in
\cite{calabimetric} where its Levi-Civita covariant derivative is computed, it
is proved that it is of constant sectional curvature, that $\H$ is then isometric
to a portion of a sphere in $C^\infty(M)$ and that both the Dirichlet problem (find a geodesic connecting two fixed points) and the Cauchy problem (find a geodesic with assigned starting point and speed) admit
smooth explicit solutions.

The gradient metric was introduced and studied in \cite{calabimetric,calazheng}.
Its Levi-Civita connection, geodesic equation and curvature are written down in \cite{calazheng}.
In this paper, we continue to study its geometry. We solve the Cauchy problem of its geodesic equation, so we prove it is locally well-posed, unlike the corresponding problem for the $L^2$ metric, which is known to be ill-posed. 

Actually, we define a more general metric, the linear combination of the three metrics on $\H$ we call \emph{combination metric} whose special instance is the  \emph{sum metric}, i.e. the sum of the gradient and Calabi metrics.

We denote the H\"older spaces with respect to the fixed K\"ahler metric $g$ by $C^{k,\alpha}(g)$. We prove that our Cauchy problem is well-posed (See Thm.~ \ref{thm:mainthmgenkahl}, \ref{thm:smoothsum} and \ref{thm:mainthmgradient}).
\begin{thm}\label{thm: sum metric main theorem rough statement}
On a compact K\"ahler manifold, for every initial K\"ahler potential $\phi_0$,
 and initial speed $\psi_0$ in $C^{k,\alpha}(g)$, for all $k\geq 6$ and  $\alpha \in (0,1)$, there exists, for a small time $T$, a unique $C^2([0,T],C^{k,\alpha}(g)\cap \mathcal H)$ geodesic for the combination metric, starting
from $\phi_0$ with initial velocity $\psi_0$. 
 Moreover if $(\phi_0, \psi_0)$ are smooth then also the solution is.
\end{thm}

Furthermore, we prove a Rauch type comparison theorem of the Jacobi fields (Theorem \ref{rauchcomparisontheorem}) between the gradient metric and the Calabi metric.
\begin{thm}Let $\gamma_G$ and $\gamma_C$ be two geodesics of equal length with respect to the gradient metric and the Calabi metric respectively and suppose that for every $X_G\in T_{\gamma_G(t)}\mathcal H$ and $X_C\in T_{\gamma_C(t)}\mathcal H$, we have 
$$K_G(X_G,\gamma_G'(t))\leq K_C(X_C,\gamma_C'(t)).$$
Let $J_G$ and $J_C$ be the Jacobi fields along $\gamma_G$ and $\gamma_C$ such that
\begin{itemize}
\item $J_G(0)=J_C(0)=0$,
\item $J'_G(0)$ is orthogonal to $\gamma'_G(0)$ and $J'_C(0)$ is orthogonal to $\gamma'_C(0)$ ,
\item $\|J'_G(0)\|=\|J_C'(0)\|.$
\end{itemize}
then we have, for all $t\in [0,T]$,
$$\|J_G(t)\|\geq\frac{\left|\sin \left(2t\sqrt{\vol}\right)\right|}{\sqrt{\vol}}.$$
\end{thm}

The sum metric arises from Sasakian geometry. Indeed the geometric motivation comes naturally from the space of Sasakian metrics $\H_S$ as follows.

Since $\H_K$ naturally embeds in the Ebin space $\M$, it is natural to ask what the restriction of the Ebin metric is. To our knowledge,
the restriction of Ebin metric to subspaces of the space of Riemannian metrics was first considered
by \cite[(9.19), page 2485]{Smolentsev} (for the space of K\"ahler metrics, see \cite{clarkerubinstein}).
In this paper we consider on $\H_K$ the metric given by (twice) the sum of the Calabi and the gradient metric and we will refer to it as the sum metric. Its study is justified by the fact that it arises when restricting the Ebin metric to the space of \emph{Sasakian} metrics, introduced (and endowed with the Sasakian analogue of the Mabuchi metric) in \cite{guanzhang1, guanzhangreg}.% (see also \cite{weiyong}).

One of our results is the following.
\begin{prop}\label{restriction}
The restriction of the Ebin metric of $\M$ to the space of Sasakian metrics is twice the sum metric.
\end{prop}

Moreover, Theorem \ref{thm: sum metric main theorem rough statement}
can be generalized to the Sasakian setting, leading to the corresponding statement for the restriction of the Ebin metric to the space of Sasakian metrics.

The paper is organized as follows. In section \ref{sec:prelim} we recall the main definition of the space of K\"ahler metrics and in section \ref{sec:comb} we write down the Levi-Civita connection of the combination metric and study the equation of the Cauchy problem for gradient metric.  
Finally, in section \ref{sec:sasakian} we compute the restriction of the Ebin metric on the space of Sasakian metrics, proving Prop.~\ref{restriction}.

\subsection*{Acknowledgements} The authors are grateful to Prof. Xiuxiong Chen and Prof. Fabio Podest\`a for
encouragement and support, to Prof. Qing Han for his insights and explanations and to Dario Trevisan, Prof. Elmar Schrohe and Prof. Christoph Walker for helpful discussions.

We finally thank Boramey Chhay who let us know that he independently proved Proposition \ref{restriction} and Prof. Stephen Preston for sharing his insightful knowledge of infinite-dimensional geometry.

S.~C. is supported by the Project PRIN ``Variet\`a reali e complesse: geometria, topologia e analisi armonica'', 
by SIR 2014 AnHyC ``Analytic aspects in complex and hypercomplex geometry" (code RBSI14DYEB),
and by GNSAGA of INdAM.

D.~P. is supported by the Research Training Group 1463 ``Analysis, Geometry and String Theory'' of the DFG, as well as the GNSAGA of INdAM.

The work of K.~Z. has received funding from the European Union's Horizon 2020 
research and innovation programme under the Marie Sk{\l}odowska-Curie grant agreement No 703949, 
and was also partially supported by the Engineering and Physical Sciences Research Council (EPSRC) on a Programme Grant entitled 
``Singularities of Geometric Partial Differential Equations" reference number EP/K00865X/1.

\section{Preliminaries} \label{sec:prelim}
In this section we recall the definitions of space of Riemannian and K\"ahler metrics and several weak Riemannian structures on them.
\subsection{Ebin metric}
The space of the Riemmanian metrics $\M$ is identified with the space $S^2_+(T^*M)$ of all symmetric positive
$(0,2)$-tensors on $M$. The formal tangent space at a metric $g \in \M$ is then given by all symmetric $(0,2)$-tensors
$S^2(T^*M)$.
For $a, b \in T_{g} \M$, the Ebin \cite{ebin} metric is defined as the pairing
\[
g_E(a, b)_{g} = \int_M g(a, b) dv_{g}
\]
where $g(a, b)$ is the metric $g$ extended to $(0,2)$-tensors and $dv_{g}$ is the volume form of $g$. From e.g. \cite{michor} one can see that the curvature is non-positive and the geodesic satisfies the equation
\begin{align*}
g_{tt} = g_t g^{-1} g_t + \frac 1 4 \tr (g^{-1} g_t g^{-1} g_t) g - \frac 1 2 \tr(g^{-1} g_t)g_t.
\end{align*}
Moreover in \cite{michor} the explicit expression of the Cauchy geodesics is given.

\subsection{Space of K\"ahler potentials}
Moving on to K\"ahler manifolds, let $(M, \omega, g)$ be a compact K\"ahler manifold of complex dimension $n$, with $\omega$ a K\"ahler form and $g$ the associated K\"ahler metric. By the $\de \debar$-Lemma, the space of all K\"ahler
metrics cohomologous to $\omega$ can be parameterized by K\"ahler
potentials; namely, one considers the space
$\H$ of all smooth real-valued $\phi$ such that
\[
 \omega_\phi:= \omega + i \de \debar \phi > 0
\]
and satisfy
the normalization condition \cite{donaldsonmetric}
\begin{equation} \label{equaI}
I(\phi) := \int_M\phi\frac{\omega^n}{n!}- \sum_{i=0}^{n-1}\frac{i+1}{n+1}
\int_{M}\partial\phi\wedge\bar\partial\phi\wedge\frac{\omega^{i}}{i!}\wedge\frac{\omega^{n-1-i}_{\phi}}{(n-1-i)!}= 0.
\end{equation}
The tangent space of $\H$ at $\phi$ is then given by
\[
T_\phi \H = \biggl  \{ \psi \in C^\infty(M): \int_M \psi \frac{\omega_\phi^n}{n!} = 0 \biggr
\}.
\]

\subsection{Donaldson-Mabuchi-Semmes's $L^2$-metric}
Donaldson, Mabuchi and Semmes \cite{donaldsonmetric, mabuchi, semmes} defined a pairing on the tangent
space of $\H$ at $\phi$ given by
\begin{align}\label{dms}
g_M(\psi_1,\psi_2)_\phi= \int_M \psi_1 \psi_2 \frac{\omega_{\phi}^{n}}{n!}.
\end{align}

We shall refer to this metric as the \emph{$L^2$-metric}. It makes $\H$ a non-positively curved, locally symmetric space. A geodesic $\phi$ satisfies
\begin{align}\label{equa:dmsgeodesic}
\phi'' - \frac 1 2 |d\phi'|_\phi^2 =0
\end{align}
where $|d \phi|_\phi^2$ denotes the square norm of the gradient of $\phi'$ with respect to the metric $\omega_\phi$. The geodesic equation can be written down as a degenerate complex Monge-Amp\`ere equation.
It was proved by Chen \cite{chen_C11} that there is a $C^{1,1}$ solution for the Dirichlet problem. More work on this topic was done in \cite{MR2040638, berman_demailly, blockigeod, conicalgeodesic,darvas_lempert,MR1959581,  lempert_vivas, MR2660461}, which is far from a complete list.
\subsection{Space of conformal volume forms}\label{Space of conformal volume forms}
According to the Calabi-Yau theorem, there is a bijection between $\H$
and the space of \emph{conformal volume forms}
\begin{equation} \label{spaceC}
\call C = \biggl \{ u \in C^\infty(M): \int_M e^u \frac{\omega^{n}}{n!} = \vol \biggr \}
\end{equation}
that is the space of positive smooth functions on $M$ whose integral with respect to the initial measure is
equal to the volume of $M$ (which is constant for all metrics in $\H$).
The map is given by $$\H \ni \phi \mapsto \log \frac{\omega_\phi^n}{\omega_0^n},$$ where
$ \frac{\omega_\phi^n}{\omega_0^n}$ represents the unique positive function $f$ such that
$\omega_\phi^n = f \omega_0^n$.
The tangent space $T_u \call C$ is then given by
\[
T_u \call C = \biggl \{ v \in C^\infty(M): \int_M v e^u \frac{\omega^{n}}{n!} = 0 \biggr \}.
\]
\subsection{Calabi metric}
Calabi \cite{calabi1954space} introduced the now known \emph{Calabi metric} as the pairing
\begin{equation} \label{calabi}
g_\textup{C}(\psi_1, \psi_2)_\phi = \int_M \Delta_\phi \psi_1 \Delta_\phi \psi_2 \frac{\omega_{\phi}^{n}}{n!}
\end{equation}
where, here and in the rest of the paper, the Laplacian is defined as
\[
\Delta_\phi f = (i \de \debar f, \omega_\phi)_\phi
\]
i.e. the $\debar$-Laplacian.
The geometry studied in \cite{calabimetric} is actually the one of $\call C$, where the Calabi metric has the simpler form
\begin{equation} \label{calabiC}
g_\textup{C}(v_1, v_2)_u = \int_M v_1 v_2 e^u \frac{\omega^{n}}{n!}.
\end{equation}

Back in $\H$, the geodesic equation is
\begin{align}\label{equa:calabigeodesic}
\Delta_\phi \phi'' - |i \de \debar \phi'|^2_\phi +\frac{1}{2}(\Delta_\phi \phi')^2 + \frac 1 {2 \vol} g_C(\phi', \phi')=0.
\end{align}

\subsection{Dirichlet metric}
In \cite{calabimetric, calazheng}, the \emph{Dirichlet (or gradient) metric} is defined as the pairing
\begin{equation} \label{gradient}
g_\textup{G}(\psi_1, \psi_2)_\phi = \int_M (d \psi_1, d \psi_2)_\phi d\mu_\phi
\end{equation}
that is, the global $L^2(d\mu_\phi)$-product of the gradients of $\psi_1$ and $\psi_2$.
Its geodesic equation is
\begin{align}\label{equa:gradientgeodesic}
 2 \Delta_\phi \phi ''  - |i \de \debar \phi '|_\phi^2 + (\Delta_\phi \phi ')^2 = 0
\end{align}
where $|i \de \debar \phi '|_\phi^2$ denotes the square norm with respect to $\omega_\phi$ of the $(1,1)$-form $i \de \debar \phi'$.

\section{Combination metrics} \label{sec:comb}
We can combine together the three metrics as follows. Let $\alpha, \beta,\gamma$ be three non-negative constant and at least one of them positive.
Consider the metric
\begin{equation} \label{cmbination}
g(\psi_1, \psi_2)_\phi = \alpha\cdot g_M(\psi_1, \psi_2)_\phi
+\beta \cdot g_{G}(\psi_1, \psi_2)_\phi
+\gamma\cdot g_{C}(\psi_1, \psi_2)_\phi .
\end{equation}
which will be referred to as the \emph{combination metric}.

Let us prove the existence of the Levi-Civita covariant derivative for the combination metric.
We can write \begin{equation}  g(\psi_1, \psi_2)_\phi = g_\textup{C}(M_\phi \psi_1, \psi_2). \end{equation} where
\[
 M_\phi = \alpha G_\phi^2 - \beta G_\phi + \gamma
 \]
 where $G_\phi$ is the Green operator associated to the Laplacian $\Delta_\phi$.

 We have the following.

\begin{prop} \label{prop:LCcomb}
For a curve $\phi \in \H$ and a section $v$ on it, the Levi-Civita covariant derivative of the combination metric is the unique $D_t v$ that solves
\begin{equation*}
M_\phi D_t \psi =[ G_\phi^2 \alpha D^M_t  - \beta G_\phi D^G_t  + \gamma D^C_t] \psi
\end{equation*}
where $ D^\textup{M}, D^G_t, D^C_t$ are the covariant derivatives of the $L^2$, gradient and Calabi metric.
\end{prop}
\begin{proof}
 We start by proving that $M_\phi$ is a bijection of $T_\phi \H$.
 The injectivity holds because it defines a metric. To prove surjectivity, we see that the problem $M_\phi u = f$ is equivalent to $Du = h$ where $D = \gamma \Delta_\phi^2 - \beta \Delta_\phi + \alpha$.
 It is elliptic and then by known results we have
 \[
  C^\infty(M) = \ker D \oplus \Im(D)
 \]
 and by integration and the normalization condition on $T_\phi \H$ we immediately see that $T_\phi \H \cap \ker D = 0$, so $T_\phi \H = \Im(D) \cap T_\phi \H$ and we obtain surjectivity.

 The fact that $D_t$ is torsion-free is evident from its definition. Let us now compute
 \begin{align*}
  \frac d {dt} g(\psi, \psi)	&= 2 \alpha g_M(D_t^M \psi, \psi) + 2 \beta g_G(D_t^G \psi, \psi) + 2 \gamma g_C (D_t^C \psi, \psi)\\
				&= 2 \alpha g_C(G_\phi^2 D_t^M \psi, \psi) - 2 \beta g_C (G_\phi D_t^G \psi, \psi) + 2 \gamma g_C(D_t^C \psi, \psi)\\
				&= 2 g_C( [G_\phi^2 \alpha D^\textup{M}_t  - \beta G_\phi D^G_t  + \gamma D^C_t] \psi, \psi)\\
				&= 2 g_C(M_\phi D_t \psi, \psi)\\
				&= 2 g(D_t \psi, \psi)
 \end{align*}
so the compatibility with the metric holds as well.
\end{proof}

\subsection{Geodesic equation of the combination metric}
The geodesic equation of the combination metric is the combination of the geodesic equations of $L^2$-metric, gradient metric and the Calabi metric.
After rearrangement, it is written in the following form
\begin{align}\label{combination geodesic}
 [\alpha  - \beta \Delta_\phi +\gamma \Delta_\phi ^2]\phi'' =
  \frac{\alpha}{2} |d\phi'|^2_\phi
 +\biggl [\frac{\beta}{2}-\gamma \Delta_\phi \biggr]|i \de \debar \phi'|^2_\phi
 +\biggl [\frac{\beta}{2}+\frac{\gamma}{2}\Delta_\phi \biggr](\Delta_\phi \phi')^2 .
\end{align}

The key observation is that the differential order on the both sides of the geodesic equation \eqref{combination geodesic} are the same. We will carry out in detail in the next section the study of the geodesic equation with $\beta=\gamma=1$ and $\alpha=0$, the general case with $\alpha=1$ is similar, so we omit the proof.

This observation suggest that, though the Cauchy problem of the geodesic ray with respect to the $L^2$-metric is ill-posed, after combining the $L^2$-metric with the Calabi metric and the gradient metric, the new geodesic equation is well-posed.

\subsection{Local well-posedness of the geodesic equation}\label{Local well-posedness of the geodesic equation}
\subsubsection{Existence and uniqueness}
\label{Existence and uniqueness}
Recall the definition of the space of K\"ahler potentials
\[
\H=\{ \phi \in C^\infty (M): \omega + i \de \debar \phi >0, I(\phi) = 0 \}.
\]

We are aiming to solve the geodesic equation with $\beta=\gamma=1$ and $\alpha=0$, i.e. the equation
\begin{equation} \label{equa:geodesicsforsummetric}
(\Delta_\phi - I) \biggl ( ( \Delta_\phi \phi' )' + \frac 1 2 (\Delta_\phi
\phi')^2 \biggr ) - \frac 1 2 | i \de \debar \phi' |^2_\phi=0.
\end{equation}

We rewrite it as a system
\begin{equation} \label{equa:systemsum}
    \begin{cases}
    \phi' = \psi  \\
 \psi ' = L_\phi (\psi) :=
   \Delta_{\phi}^{-1}
   \biggr [
   \frac{1}{2}(\Delta_\phi - I)^{-1} | i \de \debar \psi |_\phi^2
 + | i \de \debar \psi |_\phi^2 + \frac 1 2 (\Delta_\phi \psi)^2
   \biggl ]
  \end{cases}
\end{equation}
with the initial data $
   \phi(0) = \phi_0 , \psi(0) = \psi_0 \in \Cka{k}.$

Take a constant $\delta>0$ such that $\omega + i \de \debar \phi_0 \geq 2 \delta \omega$. Let us introduce also the following
function spaces
\[
\H^{k,\alpha}= \{ \phi \in C^{k,\alpha} (g): \omega + i \de \debar \phi >0, I(\phi) = 0\}
\]
and
\[
\H_{\delta}^{k,\alpha}=\{ \phi \in C^{k,\alpha} (g): \omega + i \de \debar \phi \geq \delta\omega, I(\phi) = 0\},
\]
where $k\geq 2$ and $\alpha\in (0, 1)$.

The aim of this subsection is to prove the following.

\begin{thm}\label{thm:mainthmgenkahl}
 For every integer $k\geq 6$ and $\alpha \in (0, 1)$ and initial data $\phi_0\in \H^{k,\alpha}_{\delta}$ and $\psi_0 \in T_{\phi_0} \H^{k,\alpha}$ there exists a positive $\epsilon$
and
 a curve $\phi \in C^2 ((-\epsilon, \epsilon)  , \H_{\delta}^{k,\alpha} )$
which is the
 unique solution of \eqref{equa:geodesicsforsummetric} with initial data $(\phi_0, \psi_0)$.
\end{thm}

We need the following lemma.
\begin{lemma}[Schauder estimates, see {\cite[p.~463]{besse}}] \label{schauderest}
Let $P$ be an elliptic linear operator of order $2$ acting on the H\"older space $C^{k+2, \alpha}(g)$. Then for $u \in C^{k+2, \alpha}(g)$ we have
\[
 \|u \|_\Cka{k+2}  \leq c_1 \| Pu \|_\Cka{k}  + c_2 \|u\|_{L^\infty}
\]
where $c_1$ depends only on the $\Cka{k}$-norm of the coefficients of $P$ and, if $u$ is $L^2(g)$-orthogonal to $\ker P$, then $c_2=0$.
\end{lemma}

The structure of the system \eqref{equa:systemsum} suggests to consider
the following complete metric space
\begin{equation}\label{equa:functionspace}
 X=C^2 ([-\epsilon, \epsilon]  , \H_{\delta}^{k, \alpha}) \times C^2([-\epsilon, \epsilon]  , C^{k, \alpha}(g))
\end{equation}
as the function space where we are going to look for solutions of our system.
The norm that we consider is defined for $\psi \in C^2 ([-\epsilon, \epsilon],
C^{k, \alpha}(g))$ as
\begin{equation*}
 |\psi|_{k,\alpha}:= \sup_{t\in[-\epsilon, \epsilon]} \| \psi (t, \cdot )
\|_{C^{k,\alpha}(g)} ,
\end{equation*}
 and in the product space, the norm of any element
 $(\phi ,\psi ) \in X$ is
\begin{equation*}
   |(\phi, \psi)|_{k, \alpha} : = |\phi|_{k,\alpha}+ |\psi|_{k, \alpha}    .
\end{equation*}

We work in an appropriate metric ball in $X$ obtained by the following lemma.

\begin{lemma} \label{lemmar}
 There exists $r>0$ such that if $\phi \in \Cka{k}$ is such that $\| \phi - \phi_0 \|_\Cka{k} < r $ then $\phi \in \H_\delta^{k,\alpha}$.
\end{lemma}
\begin{proof}
 Being $k \geq 2$ we have  $\| \phi - \phi_0 \|_\Cka{2} \leq \| \phi - \phi_0 \|_\Cka{k} < r $.
Then
 \begin{align*}
  g_\phi 	&=	g_\phi - g_{\phi_0} +  g_{\phi_0} \\
		&\geq	-\| \phi - \phi_0 \|_\Cka{2} g + 2\delta g\\
		&\geq	(2\delta-r) g
 \end{align*}
which is strictly bigger than $\delta g$ for $r < \delta$.
\end{proof}

We consider the operator
\begin{equation} \label{equa:mapTgen}
 T(\phi, \psi) = \biggl ( \phi_0 + \int_0^t \psi (s) ds,  \psi_0 + \int_0^t (L_\phi (\psi))(s) ds \biggr ).
\end{equation}

Let us now fix $r>0$ as in Lemma \ref{lemmar}. We have the following proposition, but first let us isolate a lemma.
\begin{lemma}\label{lemmag}
There exist a positive $C$ depending only on $r$ and $g$ such that 
\[
 \| g_\phi^{a \bar b} \|_\Cka{k} \leq C.
 \]
\end{lemma}
\begin{proof}
For fixed $a,b$ it holds $\| g_\phi^{a \bar b} \|_{\Cka{k}} \leq \| g_\phi^{-1} \|_{\Cka{k}}$ where the norm is intended as operator norm.
Then by the sub-multiplicative property we have $ \| g_\phi^{-1} \|_{\Cka{k}} \leq  \| g_\phi \|^{-1}_{\Cka{k}}$ and by estimate in the proof of Lemma~\ref{lemmar} we have that $ \| g_\phi \|^{-1}_{\Cka{k}} \leq (2\delta-r)^{-1}\|g\|_{\Cka{k}} =: C(r,g)$.
\end{proof}

\begin{prop}
 For any $(\phi_0, \psi_0) \in \H_{\delta}^{k, \alpha} \times C^{k, \alpha}(g)$ there exists $\epsilon>0$ such that the metric ball $B_r(\phi_0, \psi_0) \subset X$ centered in $(\phi_0, \psi_0)$ of radius $r$ is mapped into itself by $T$.
 \end{prop}
\begin{proof}
 We need to estimate $|T(\phi_0, \psi_0)-(\phi_0, \psi_0)|_{k,\alpha}$.
 Let us estimate the first component
 \begin{align*}
  \biggl | \phi_0 + \int_0^t \psi (s) ds -\phi_0 \biggr |_{k, \alpha} &= \sup_{t\in[-\epsilon, \epsilon]} \biggl \| \int_0^t \psi (s) ds  \biggr \|_\Cka{k}  \\
								      &\leq   \sup_{t\in[-\epsilon, \epsilon]} \int_0^t \| \psi (s) \|_\Cka{k}  ds \\
								      &\leq \sup_{t\in[-\epsilon, \epsilon]} \int_0^t \sup_{s\in[-\epsilon, \epsilon]} \| \psi (s)\|_\Cka{k} ds \\
								      &\leq \epsilon \cdot ( |\psi_0|_{k,\alpha} + |\psi - \psi_0|_{k,\alpha} ) \\
								      &\leq \epsilon \cdot ( |\psi_0|_{k,\alpha} + r ).
\end{align*}

As for the second component, it is clear it is enough to estimate $\| L_\phi(\psi) \|_\Cka{k} $ for every $t$.

We have, by Lemma \ref{schauderest},
\begin{align*}
 \| L_\phi(\psi) \|_\Cka{k}  &\leq \biggl \| \Delta_{\phi}^{-1}\biggr [ \frac{1}{2}(\Delta_\phi - I)^{-1} | i \de \debar \psi |_\phi^2 + | i \de \debar \psi |_\phi^2 + \frac 1 2 (\Delta_\phi \psi)^2   \biggr ] \biggr  \|_\Cka{k}  \\
				    &\leq C( \|\phi\|_\Cka{k} ) \biggl \| \frac{1}{2}(\Delta_\phi - I)^{-1} | i \de \debar \psi |_\phi^2 + | i \de \debar \psi |_\phi^2 + \frac 1 2 (\Delta_\phi \psi)^2 \biggr \|_\Cka{k-2} .
\end{align*}
To estimate the first summand we have
\begin{align*}
 \biggl \| \frac{1}{2}(\Delta_\phi - I)^{-1} | i \de \debar \psi |_\phi^2 \biggr \|_\Cka{k-2} 	&\leq C( \|\phi\|_\Cka{k-2} ) \| | i \de \debar \psi |_\phi^2 \|_\Cka{k-4}  \\
													&\leq C(r) \| g_\phi^{i \bar \jmath} g_\phi^{k \bar \ell} \psi_{i \bar \ell} \psi_{k \bar \jmath} \|_\Cka{k-4}  \\
													&\leq C(r)  \| \psi \|_\Cka{k}
\end{align*}
where in the first inequality we have used again Lemma \ref{schauderest} and in the last we have used that $\| \psi \|_\Cka{k-2} \leq \| \psi \|_\Cka{k} < r$.

The second summand is estimated, similarly as before, by
\begin{equation*}
 \| | i \de \debar \psi |_\phi^2 \|_\Cka{k-2} \leq C(r) \| \psi \|_\Cka{k}.
\end{equation*}

The third summand is
\begin{align*}
 \biggl \| \frac 1 2 (\Delta_\phi \psi)^2 \biggr \|_\Cka{k-2} \leq \| (\Delta_\phi \psi)^2 \|_\Cka{k-2}	&\leq \| g_\phi^{i \bar \jmath} \psi_{i \bar \jmath} \|_\Cka{k-2}^2 \\
													&\leq C(r) \| \psi \|_\Cka{k}.
\end{align*}

So we can conclude that the second component of $|T(\phi_0, \psi_0)-(\phi_0, \psi_0)|_{k,\alpha}$ is estimated by $\epsilon C(r) | \psi - \psi_0|_{k, \alpha} \leq \epsilon r C(r)$, so it is enough to choose $\epsilon(r)$ such that $\epsilon(r) C(r) < 1$.
\end{proof}

Our second step is the following.
\begin{prop}
 The map $T$ on the metric ball $B_r(\phi_0, \psi_0)$ is a contraction.
\end{prop}

\begin{proof}
For $(\phi, \psi)$ and $(\tilde \phi, \tilde\psi)$  in $B_r(\phi_0, \psi_0)$, let for simplicity $\tilde L = L_{\tilde \phi}$.
We need to estimate $\| L(\psi) - \tilde L(\tilde \psi) \|_\Cka{k}$. Define $f$ and $\tilde f$ such that $L(\psi) = \Delta_\phi^{-1}f$ and $\tilde L (\tilde \psi) = \Delta_{\tilde \phi}^{-1} \tilde f$.
Then we have
\begin{align*}
 \Delta_\phi ( L(\psi) - \tilde L (\tilde \psi) )	&= f - \tilde f - \Delta_\phi \tilde L(\tilde \psi) + \Delta_{\tilde \phi} \tilde L (\tilde \psi)\\
							&= f - \tilde f + (g_{\tilde \phi}^{i \bar \jmath} - g_\phi^{i \bar \jmath}) (\tilde L(\tilde \psi))_{i \bar \jmath}
\end{align*}
so by the Schauder estimates of Lemma \ref{schauderest} we have
\begin{align*}
 \| L(\psi) - \tilde L(\tilde \psi) \|_\Cka{k} 	&\leq  C( \|\phi\|_\Cka{k}) \\
						& \cdot \biggl ( \| f-\tilde f \|_\Cka{k-2} + \| \Delta_{\tilde \phi} \tilde L (\tilde \psi) - \Delta_\phi \tilde L(\tilde \psi) \|_\Cka{k-2} \biggr ).
\end{align*}
To estimate the second summand, let $g_s = (1-s) g_\phi + s g_{\tilde \phi}$. Then we notice we can write
\[
 \Delta_{\tilde \phi} \tilde L (\tilde \psi) - \Delta_\phi \tilde L(\tilde \psi) = - \biggl (\int_0^1 g_s^{i \bar \ell} g_s^{k \bar \jmath} ds \biggr ) \cdot (\tilde \phi - \phi)_{k \bar \ell} \cdot (\tilde L(\tilde \psi))_{i \bar \jmath}.
\]
so we have
\begin{align*}
 \| \Delta_{\tilde \phi} \tilde L (\tilde \psi) - \Delta_\phi \tilde L(\tilde \psi) \|_\Cka{k-2}	&\leq C( \|\phi\|_\Cka{k},  \|\tilde \phi\|_\Cka{k}) \| \tilde \phi - \phi \|_\Cka{k} \cdot \| \tilde L \tilde \psi \|_\Cka{k} \\
													&\leq C(r)  \| \tilde \phi - \phi \|_\Cka{k}
\end{align*}
where in the last inequality we have used the estimate for  $\| \tilde L \tilde \psi \|_\Cka{k}$ from the previous proposition.

Let us now consider $\tilde f - f$ which can be written as
\begin{align} \label{fftilde}
\tilde f - f	&= \frac 1 2 (\Delta_\phi -1)^{-1} | i \de \debar \psi|_\phi^2  - \frac 1 2 (\Delta_{\tilde \phi} -1)^{-1} | i \de \debar \tilde \psi|_{\tilde \phi}^2 \\ \nonumber
		&+ | i \de \debar \psi|_\phi^2 - | i \de \debar \tilde \psi|_{\tilde \phi}^2 \\
		&- \frac 1 2 (\Delta_\phi \psi)^2 + \frac 1 2 (\Delta_{\tilde \phi} \tilde \psi)^2. \nonumber
\end{align}

Let $h - \tilde h$ be the first summand, so we can write
\[
 (\Delta_\phi -1)(h-\tilde h) = | i \de \debar \psi|_\phi^2 - |i \de \debar \tilde \psi|_{\tilde \phi}^2 + (\Delta_{\tilde \phi} - \Delta_\phi) \tilde h.
\]

Again by Lemma \ref{schauderest} we have
\[
\| h - \tilde h \|_\Cka{k-2} \leq C( \| \phi\|_\Cka{k-2}) \cdot \| | i \de \debar \psi|_\phi^2 - | i \de \debar \tilde \psi|_{\tilde \phi}^2 + (\Delta_{\tilde \phi} - \Delta_\phi) \tilde h \|_\Cka{k-4}.
\]

The second summand is
\begin{align*}
  \|(\Delta_{\tilde \phi} - \Delta_\phi) \tilde h \|_\Cka{k-4}	&\leq \biggl \|- \biggl (\int_0^1 g_s^{i \bar \ell} g_s^{k \bar \jmath} ds \biggr ) \cdot (\tilde \phi - \phi)_{k \bar \ell} \cdot \tilde h_{i \bar \jmath} \biggr \|_\Cka{k-4} \\
								&\leq C(\|\phi\|_\Cka{k-2}, \|\tilde \phi\|_\Cka{k-2}) \cdot \| \tilde \phi - \phi \|_\Cka{k-2} \cdot \| \tilde h\|_\Cka{k-2}.
\end{align*}

By definition of $\tilde h$ we estimate then
\begin{align*}
 \|\tilde h\|_\Cka{k-2} &\leq C(\| \tilde \phi\|_\Cka{k-2}) \cdot \| | i \de \debar \tilde\psi|_{\tilde \phi}^2 \|_\Cka{k-4}\\
			&\leq C(r) (\| \tilde \phi\|_\Cka{k-2}+1)^2 \cdot \| \tilde \psi\|_\Cka{k-2}^2\\
			&\leq C(r).
\end{align*}
So we finally have for the first summand in \eqref{fftilde}
\[
 \biggl \| \frac 1 2 (\Delta_\phi -1)^{-1} | i \de \debar \psi|_\phi^2  - \frac 1 2 (\Delta_{\tilde \phi} -1)^{-1} | i \de \debar \tilde \psi|_{\tilde \phi}^2 \biggr \|_\Cka{k-2} \leq C(r) ( \|\tilde \phi - \phi\|_\Cka{k-2} + \| \tilde \psi - \psi \|_\Cka{k-2}).
\]

The second summand in \eqref{fftilde} is estimated by the same trick as in the previous proposition.

For the last summand in \eqref{fftilde} we have
\begin{align*}
  \frac 1 2 (\Delta_{\tilde \phi} \tilde \psi)^2- \frac 1 2 (\Delta_\phi \psi)^2 &= \frac 1 2 ( \Delta_\phi \psi - \Delta_{\tilde \phi}\tilde \psi)( \Delta_\phi \psi + \Delta_{\tilde \phi}\tilde \psi)\\
										 &= \frac 1 2 ( \Delta_\phi \psi - \Delta_{\tilde \phi}\psi + \Delta_{\tilde \phi}\psi - \Delta_{\tilde \phi}\tilde \psi)( \Delta_\phi \psi + \Delta_{\tilde \phi}\tilde \psi).
\end{align*}
so we estimate
\begin{align*}
 \biggl \|  \frac 1 2 (\Delta_{\tilde \phi} \tilde \psi)^2- \frac 1 2 (\Delta_\phi \psi)^2 \biggr \|_\Cka{k-2}	&\leq \biggl ( \| \Delta_\phi \psi - \Delta_{\tilde\phi} \psi \|_\Cka{k-2} + \| \Delta_{\tilde \phi}\psi - \Delta_{\tilde \phi}\tilde \psi\|_\Cka{k-2} \biggr )\\
														&\cdot \biggl ( \|  \Delta_\phi \psi \|_\Cka{k-2} + \| \Delta_{\tilde \phi}\tilde \psi \|_\Cka{k-2} \biggr ).	
\end{align*}

By the estimates for the Laplacians we are able to say that this quantity is $\leq C(r) ( \|\tilde \phi - \phi\|_\Cka{k} + \| \tilde \psi - \psi \|_\Cka{k})$.

Again, the estimate for the norm $| \cdot |_{k,\alpha}$ is the same multiplied by $\epsilon$, so again it suffices to pick $\epsilon(r)$ such that $\epsilon(r)C(r) < 1$.
\end{proof}

\subsubsection{Higher regularity}
Now we explain how to obtain the smoothness of the solution of Theorem \ref{thm:mainthmgenkahl}.
\begin{thm} \label{thm:smoothsum}
 For every $\phi_0\in \H$ ,
 $\psi_0 \in T_{\phi_0} \H$ ,
  there exists a positive $\epsilon$
and
 a curve $\phi \in C^\infty ((-\epsilon, \epsilon)  , \H)$
 which is the
 unique solution of \eqref{equa:geodesicsforsummetric} with smooth initial data $(\phi_0, \psi_0)$.
\end{thm}

We isolate the following technical lemma that can be proved by computation in local coordinates.
\begin{lemma} \label{lemmaAderivatives}
 Let $\de_A$ be the derivative with respect to the complex coordinate $z_A$ and let $f_A = \de_A f$. Then the following hold
 \begin{align*}
 ( g_\vphi^{i\bar \jmath})_A&=-g_\phi^{i \bar s} ({g_\phi}_{ \bar s m})_A g_\phi^{m \bar \jmath};\\
  \de_A (\Delta_\phi f) 		&= \Delta_\phi f_A +( g_\vphi^{i\bar \jmath})_A f_{i\bar \jmath}; \\
  \de_A | i \de \debar \psi|^2_\phi	&=2 (i \de \debar \psi, i \de \debar \psi_A) - \psi_{i \bar \jmath} \psi_{k \bar \ell} g_\phi^{i \bar s} ({g_\phi}_{ \bar s m})_A g_\phi^{m \bar \ell} g_\phi^{m \bar \jmath} -  \psi_{i \bar \jmath} \psi_{k \bar \ell} g_\phi^{i \bar \ell} g_\phi^{k \bar s} ({g_\phi}_{ \bar s m})_A g_\phi^{m \bar \jmath} \\
					&= 2(i \de \debar \psi, i \de \debar \psi_A) + B_{\phi \psi} \phi_A; \\
  \de_A (\Delta_\phi \psi)^2		&= 2 \Delta_\phi \psi[ \Delta_\phi \psi_A+( g_\vphi^{i\bar \jmath})_A \psi_{i\bar \jmath}]
  \end{align*}
  where $B_{\phi \psi}$ is a linear operator.
\end{lemma}

We want to derive the second equation of \eqref{equa:linsyst} by deriving the equation
\begin{equation} \label{eqtoderive_sum}
 F(\phi,\psi)=(\Delta_\phi-1) \biggl [ \Delta_\phi \psi' - |i \de \debar \psi|_\phi^2 + \frac 1 2 (\Delta_\phi \psi)^2 \biggr ] - \frac 1 2 | i \de \debar \psi |_\phi^2=0.
\end{equation}
\begin{lemma}\label{linear oparator}
$\p_A F(\phi,\psi)$ is a linear fourth order operator on $(\phi_A, \psi_A)$. When $(\phi,\psi)$ are $C^{k,\alpha}$, the coefficients of $\p_A F(\phi,\psi)$ are $C^{k-4,\alpha}$.
\end{lemma}
\begin{proof}

The derivative of the first term is, by Lemma \ref{lemmaAderivatives},
\[
 \de_A (\Delta_\phi - 1) \Delta_\phi \psi' = (\Delta_\phi -1) \bigl [ \Delta_\phi \psi_A' +(g_\phi^{i\bar \jmath})_A\psi'_{i\bar \jmath})\bigr ] + (g_\phi^{i\bar \jmath})_A(\Delta_\phi \psi')_{i\bar \jmath}
\]
where we notice linearity with respect to $\phi_A$ and $\psi_A$.

The derivative of the second term is
\begin{align*}
 \de_A (\Delta_\phi -1)| i \de \debar \psi |_\phi^2	&= (\Delta_\phi -1) \de_A | i \de \debar \psi |_\phi^2+ (g^{i\bar \jmath}_\phi)_A(| i \de \debar \psi |_\phi^2)_{i\bar \jmath} \\
							&=(\Delta_\phi -1) [2(i \de \debar \psi, i \de \debar \psi_A) + B_{\phi \psi} \phi_A]+ (g^{i\bar \jmath}_\phi)_A(| i \de \debar \psi |_\phi^2)_{i\bar \jmath}  
							\end{align*}
and we notice again linearity with respect to $\phi_A$ and $\psi_A$.

The third and fourth terms are as in Lemma \ref{lemmaAderivatives} and are linear with respect to $\phi_A$ and $\psi_A$ as well.
\end{proof}

\begin{proof}[Proof of Theorem~\ref{thm:smoothsum}]
When we are given a smooth initial data $(\phi_0,\psi_0)$ and H\"older exponent $(k,\alpha)$ with $k\geq 6$ and $\alpha \in (0,1)$, according to Theorem \ref{thm:mainthmgenkahl}, we have a maximal lifespan $\epsilon=\epsilon(k+1,\alpha)$ of the geodesic $\phi(t) \in C^2 ((-\epsilon, \epsilon)  , \H^{k+1,\alpha} )$. Meanwhile, for a less regular space $(k,\alpha)$, we have an other maximal lifespan $\epsilon(k,\alpha)$. In general, $$\epsilon(k+1,\alpha)\leq \epsilon(k,\alpha).$$ Now we explore the important property of our geodesic equation and thus prove the inverse inequality $\epsilon(k+1,\alpha)\geq \epsilon(k,\alpha)$.

Recall that our geodesic equation could be written down as a couple system \eqref{equa:systemsum}.

The important observation is that this system is of order zero. In a local coordinate chart on $M$, we take the derivative $\partial_A= \frac{\partial}{\partial z_A}$ on the both side of the equations and get
\begin{equation} \label{equa:linsyst}
  \begin{cases}
   ( \partial_A\phi)' = \partial_A\psi  \\
 (\partial_A\psi) ' = \partial_A(L_\phi \psi) .
   \end{cases}
\end{equation}

If we manage to prove that this is a linear system in $\phi_A = \de_A \phi$ and $\psi_A=\de_A \psi$ (all other functions treated as constants) then we can argue as follows.
According to Lemma \ref{linear oparator}, the coefficients of \eqref{equa:linsyst} are $C^{k-4,\alpha}$ and exist for $|t| < \epsilon(k, \alpha)$. Because of its linearity and of fourth order on $(\phi_A, \psi_A)$, its $C^{k, \alpha}$ solution $(\phi_A, \psi_A)$ exists as long as the coefficients do, so we have that $\phi$ is $C^{k+1, \alpha}$ at least for $|t|< \epsilon(k, \alpha)$, proving that $\epsilon(k+1,\alpha)\geq \epsilon(k,\alpha)$.
\end{proof}
\subsection{Exponential map, Jacobi fields and conjugate points}
\label{Jacobi equation and conjugate points}
With the local well-posedness of the geodesic, we are able to define the exponential map locally at point $\phi\in  \H$ by
\begin{align}
\exp_\phi (t\psi)=\gamma(t),  0\leq t\leq\epsilon
\end{align}
where $\gamma$ is the geodesic starting from $\phi$ with initial speed $\psi$.
Furthermore, we have the following.
\begin{corol}
For any $\phi_1\in  \H$, there exists an $\epsilon > 0$ so that for any $\phi_2\in \H $ with $\|\phi_1-\phi_2\|_{C^{2,\alpha}} < \epsilon$, there is a unique geodesic connecting $\phi_1$ to $\phi_2$ whose length is less than $\epsilon$.
\end{corol}

Now that we have achieved the existence of smooth short-time geodesics we can move a step further to bring the definition of its Jacobi vector fields.
The very definition comes from classical Riemannian geometry, see \cite{calabimetric} for more details.

Let $\gamma : [0,\epsilon)\rightarrow \mathcal{H}$ be a smooth geodesic for the metric connection $D$ on $\H$.
A Jacobi field $J$ along $\gamma$ is a map
$J : [0, \epsilon ) \rightarrow T\mathcal{H}$ such that $J(t) \in T_{\gamma (t)} \mathcal{H} $ for all $t\in [0,\epsilon)$ and moreover satisfies the Jacobi equation
\begin{equation} \label{jacobieq}
\frac{D^2}{dt^2}J(t) + R \biggl ( J(t), \frac{d}{dt}\gamma(t) \biggr ) \frac{d}{dt}\gamma(t) = 0.
\end{equation}
The Jacobi field is a vector field along the geodesic $\gamma(t)$. Let $v=\frac{d}{dt}\vert_{t=0}\gamma(t)$ at $\gamma(0)=\phi$, the geodesic is given by the exponential map $\gamma(t) = \exp_\phi tv$.
Then given $w\in T_\phi \H$, the solution of the Jacobi equation \eqref{jacobieq} with initial condition $J(0)=0$ and $J'(0) = w$ is given by
\[
 J(t)= d \exp_\phi |_{tv} tw .
\]

 The definition of conjugate points in the infinite dimensional setting is different from the one from classical Riemannian geometry.
 Let $\phi \in \H$, $\psi \in T_\phi \H$ and let $\gamma$ be the geodesic with $\gamma(0) = \phi$ and $\gamma'(0) = \psi$.
 There are two notions related to conjugate points, cf. e.g. \cite{grossman, preston_diffeo, preston_fredholm}. 
 \begin{defi} \label{defi:mono and epiconjugate}
 We say that $\gamma(1)$ is
 \begin{itemize}
\item \emph{monoconjugate} to $\phi$ if $d \exp_\phi |_\psi$ is not injective;
\item \emph{epiconjugate} to $\phi$ if $d \exp_\phi |_\psi$ is not surjective.
 \end{itemize} 
 \end{defi}
 \begin{rem}
In order to understand the conjugate points, it turns out to further study whether $d \exp_\phi |_\psi$ is a Fredholm operator between the tangent spaces of $\mathcal H$. Then the infinite dimensional version of Sard's theorem applies \cite{MR0185604}.
\end{rem}

\subsection{Dirichlet metric and a comparison theorem} \label{section:gradientmetric}
Now we continue the study of the (Dirichlet) gradient metric.
\subsubsection{Sectional curvature for the gradient metric}
We denote $\phi = \phi (s , t)$ be
a smooth two parameter family of curves in the space of K\"ahler metrics $\mathcal{H}$,
and the corresponding two parameter families of curves of tangent vectors $\phi_t$, $\phi_s$ along $\phi$
are $\mathbb{R}$-linearly independent.
The sectional curvature of the gradient metric is computed in \cite{calazheng},
\begin{align*}
  K_G (\phi_s , \phi_t )_\phi  
= \frac{1}{2}\int_M |d a(s,t)|_{g_\phi}^2\frac{\omega_\phi^n}{n!} - \frac{1}{2}
\int_M (d a(s,s) , d a(t,t))_{g_\phi} \frac{\omega_\phi^n}{n!}\, , 
\end{align*}
where the symmetric expression $a(\sigma , \tau)$ satisfies
\begin{align*}
\Delta_\phi a(\sigma , \tau) = \Delta_\phi \phi_\sigma \Delta_\phi \phi_\tau - (i \p \bar \p \phi_\sigma, i\p \bar \p \phi_\tau) .
\end{align*}

We let
\[
 \{ \phi_s , \phi_t \}_\phi =
\frac{\sqrt{-1}}{2}
\left(
g^{i\bar{\jmath}} \frac{\partial \phi_s}{ \partial z^i }\frac{\partial \phi_t}{ \partial z^{\bar{\jmath}}}
-
g^{i\bar{\jmath}} \frac{\partial \phi_t}{ \partial z^i }\frac{\partial \phi_s}{ \partial z^{\bar{\jmath}}}
\right)
=
\Ima(\partial \phi_s , \overline{\partial} \phi_t )_\phi \; .
\]
The expression of the sectional curvature $K_M$ for the $L^2$ metric is, for all linearly independent
sections $\phi_s , \phi_t$,
\[
 K_M (\phi_s , \phi_t )_\phi = -\frac{\int_M \Ima (\partial \phi_s , \overline{\partial} \phi_t)_\phi^2
\frac{\omega_\phi^n}{n!}}
{\sqrt{\int_M \phi_s^2 \frac{\omega_\phi^n}{n!}}
\sqrt{\int_M \phi_s^2 \frac{\omega_\phi^n}{n!}}
- \int_M \phi_s \phi_t \frac{\omega_\phi^n}{n!}}.
\]
Therefore, $K_M \leq 0$. On the other side, the first author proved that, for any linearly
independent sections $\phi_s , \phi_t$  the sectional curvature for the Calabi metric $K_C$ is
\[
 K_C (\phi_s , \phi_t) = \frac{1}{4\vol}.
\]
In a private communication, Calabi conjectured that there exists the following relation among the sectional curvatures of $L^2$ metric, gradient metric and Calabi metric,
\begin{align*}
K_{M} \leq K_{G} < K_{C}.
\end{align*}
\begin{rem}
It would be interesting to construct examples to detect the sign of the sectional curvature of the gradient metric and determine whether this conjecture holds.
\end{rem}

\subsubsection{Local well-posedness for the gradient metric}

On the other hand, the application of the proofs of Theorem~\ref{thm:mainthmgenkahl} and \ref{thm:smoothsum} leads to the corresponding theorem of the gradient metric.

\begin{thm}\label{thm:mainthmgradient}
For every integer $k\geq 6$ and $\alpha \in (0, 1)$ and initial data $\phi_0\in \H^{k,\alpha}$ and $\psi_0 \in T_{\phi_0} \H^{k,\alpha}$ there exists a positive $\epsilon$
and
 a curve $\phi \in C^2 ((-\epsilon, \epsilon)  , \H^{k,\alpha} )$
which is the
 unique solution of the geodesic equation \eqref{equa:gradientgeodesic} with initial data $(\phi_0, \psi_0)$.
Moreover, if the initial data is smooth, then the solution $\phi$ is also smooth.
\end{thm}

\subsubsection{Sectional curvature and stability}

The idea that the sign of the sectional curvature could be used to predict the stability of the geodesic ray goes back to Arnold \cite{MR0202082}. Intuitively, when the sectional curvature is positive, all Jacobi fields are uniformly bounded, then under a small perturbation of the initial velocities, the geodesics remain nearby. When the sectional curvature is negative, the Jacobi fields grow exponentially in time, then the geodesic rays grow unstable. When the sectional curvature is zero, the geodesic ray is linear. For the gradient metric, the picture might be more complicated as the sign of the sectional curvature might vary along different planes. 
However, we are able to examine the growth of Jacobi fields along geodesics by applying the comparison theorem for infinite dimensional manifolds.

Then with the definitions of the Jacobi equation and conjugate points in Section \ref{Jacobi equation and conjugate points}, we could apply Biliotti's \cite{MR2113672} Rauch comparison theorem for weak Riemannian metrics, see \cite{MR3138478}. 
\begin{thm}\label{rauchcomparisontheorem}Let $\gamma_G$ and $\gamma_C$ be two geodesics of equal length with respect to the gradient metric and the Calabi metric respectively and suppose that for every $X_G\in T_{\gamma_G(t)}\mathcal H$ and $X_C\in T_{\gamma_C(t)}\mathcal H$, we have 
$$K_G(X_G,\gamma_G'(t))\leq \frac{1}{4\vol}=K_C(X_C,\gamma_C'(t)).$$
Let $J_G$ and $J_C$ be the Jacobi fields along $\gamma_G$ and $\gamma_C$ such that
\begin{itemize}
\item $J_G(0)=J_C(0)=0$,
\item $J'_G(0)$ is orthogonal to $\gamma'_G(0)$ and $J'_C(0)$ is orthogonal to $\gamma'_C(0)$ ,
\item $\|J'_G(0)\|=\|J_C'(0)\|.$
\end{itemize}
then we have, for all $t\in [0,T]$,
$$\|J_G(t)\|\geq\frac{\left|\sin \left(2t\sqrt{\vol}\right)\right|}{\sqrt{\vol}}.$$
\end{thm}
\begin{proof}
In Biliotti's Rauch comparison theorem, it is required that $J_C(t)$ is nowhere zero in the interval $(0,T]$ and if $\gamma_C$ has most a finite number of points which are epiconjugate but not monoconjugate in $(0,T]$, this condition is satisfied for the Calabi metric, see \cite{calabimetric}.
Therefore the conclusion of the comparison theorem is that, for all $t\in [0,T]$,
\[
\|J_G(t)\|\geq \|J_C(t)\|.
\]
We know that, as an application of \cite[Theorem 8]{calabimetric}, that 
\[
\|J_C(t)\|=\frac{\left|\sin \left(2t\sqrt{\vol}\right)\right|}{\sqrt{\vol}},
\]
thus the resulting inequality in the proposition follows.
\end{proof}

\section{The space of Sasakian metrics} \label{sec:sasakian}
\subsection{The restricted Ebin metric}
Since the sum metric arises in the context of Sasakian geometry, in this subsection we recall the
definitions of the case.
A Sasakian manifold is a $(2n+1)$-dimensional $M$ together with a contact form $\eta$, its Reeb
field $\xi$, a $(1,1)$-tensor field $\Phi$ and a Riemannian metric $g$ that makes $\xi$ Killing, such that
\begin{equation*}
\begin{split}
\eta(\xi) = 1, \iota_\xi d\eta = 0 \\
\Phi^2 = - \id + \xi \otimes \eta \\
g(\Phi \cdot, \Phi \cdot) = g + \eta \otimes \eta \\
d\eta = g( \Phi \cdot, \cdot) \\
N_\Phi + \xi \otimes d\eta = 0
\end{split}
\end{equation*}
where $N_\Phi$ is the torsion of $\Phi$. The first four mean that $M$ is a \emph{contact metric} manifold and the last one means it is \emph{normal}, see \cite[Chap.~6]{monoBG}.

The foliation defined by $\xi$ is called \emph{characteristic foliation}. Let $D = \ker \eta$.
It is known that $(d\eta, J = \Phi|_D)$ is a \emph{transversally K\"ahler} structure, as the second, third and fourth equation above say.

A form $\alpha$ is said to be \emph{basic} if $\iota_\xi \alpha = 0$ and $ \iota_\xi d\alpha = 0$.
A function $f \in C^\infty(M)$ is basic if $\xi \cdot f = 0$. The space of smooth basic functions on $M$
is denoted by $C^\infty_B(M)$. The transverse K\"ahler structure defines the transverse operators $\de$, $\debar$ and $d^c = \frac i 2(\debar-\de)$ acting on basic forms, analogously as in complex geometry.\footnote{This definition with the $\frac 1 2$ is classical in Sasakian geometry and differs from the convention usually used in complex geometry $d^c = i (\debar - \de)$. With this convention, the relation $dd^c = i \de \debar$ holds on basic forms.} The  form $d\eta$ is basic and its
basic class is called \emph{transverse K\"ahler class}.

 Given an initial Sasakian manifold $(M, \eta_0, \xi_0, \Phi_0, g_0)$, basic functions parameterize a
 family of other Sasakian structures on $M$ which share the same characteristic foliation and are in the
 same transverse K\"ahler class, in the following way.  We follow the notation of \cite[p.~ 238]{monoBG}.

Let $\phi \in C^\infty_B(M)$ and define $\eta_\phi = \eta_0 + d^c \phi$.
The space of all $\phi$'s is
\[
\tilde  \H_S = \{ \phi \in C^\infty_B(M): \eta_\phi \wedge d\eta_\phi \neq 0 \}
\]
and, in analogy of the K\"ahler case, we consider normalized ``potentials''
\[
 \H_S =  \{ \phi \in \tilde \H_S:  I(\phi)=0 \}.
\]
The equation $I=0$ is a normalization condition, similar to \eqref{equaI}. We refer to \cite{guanzhangreg} for the definition of $I$ in our case, which is such that
\[
T_\phi  \H_S = \biggl \{ \psi \in C^\infty_B(M): \int_M \psi \Svol{\phi}{n} = 0 \biggr \}.
\]

These deformations are called of \emph{type II} and it is easy to check that they leave the Reeb foliation and the transverse holomorphic structure fixed,
since $\xi$ is still the Reeb field for $\eta_\phi$.

Every $\phi \in \H_S$ defines a new Sasakian structure where the Reeb field and the transverse holomorphic structure are the
same and
 \begin{equation} \label{defoII_tensors}
 \begin{aligned}
  \eta_\phi		&= \eta_0+d^c \phi \\
  \Phi_\phi		&= \Phi_0 - (\xi \otimes d^c \phi) \circ \Phi_0 \\
  g_\phi		&= d\eta_\phi \circ (\id \otimes \Phi_\phi) +  \eta_\phi
\otimes  \eta_\phi.
 \end{aligned}
 \end{equation}
 Note that one could write $g_\phi = d\eta_\phi \circ (\id \otimes \Phi_0) +
\eta_\phi \otimes  \eta_\phi$ since the endomorphism $\Phi_\phi - \Phi_0$  has
values parallel to $\xi$ and $d\eta_\phi$ is basic. Indeed, the range of $\Phi_\phi$ is the one of $\Phi_0$ plus a component along $\xi$, so if we contract it with $d\eta$ the latter vanishes.
As in the K\"ahler case, these deformations keep the volume of $M$ fixed, which will be denoted by $\vol$.

The $L^2$ metric was generalized to $\H_S$ in \cite{guanzhangreg,weiyong}, where Guan and Zhang solved the Dirichlet problem for the geodesic equation and He provided a Sasakian analogue of Donaldson's picture about extremal metrics.

On the space $\H_S$ one can define the Calabi metric and the gradient metric in the same ways as in formulae
\eqref{calabi} and \eqref{gradient} by using the so called \emph{basic Laplacian} which acts on basic functions
in the same way as in the K\"ahler case and by using the volume form $\Svol{\phi}{n}$ in the integrals.

In this setting, it is easy to see that the map $$\H_S \ni \phi \mapsto
\log \frac{\eta_\phi \wedge d\eta_\phi^n}{\eta_0 \wedge d\eta_0^n}$$ maps basic functions to basic functions.
The \emph{transverse Calabi-Yau theorem} of \cite{bgm} allows to prove the surjectivity of this map as in the
K\"ahler case, more precisely between $\H_S$ and the space of \emph{basic} conformal volume forms
\[
\call C_B = \biggl \{ u \in C^\infty_B(M): \int_M e^u \Svol{0}{n} = \vol \biggr \}.
\]

As noted above, the space $\call C$ can be defined also for Sasakian manifolds by just taking the
Sasakian volume form $\Svol{0}{n}$ instead of the K\"ahler one. One might ask how the spaces $\call C_B$ and $\call C$
are related. Obviously $\call C_B \subseteq \call C$ but we can say more.
\begin{prop} $\call C_B$ is totally geodesic in $\call C$.
\end{prop}
\begin{proof}
It is straightforward  to verify that for any curve in $\call C_B$ and section along it, the covariant derivative defined in \cite{calabimetric} is still basic, meaning that the (formal) second fundamental form of $\call C_B$ vanishes.
\end{proof}

%%%%%%%%%%%%%%%%%%%%%%

Let $\mathcal M$ be the Ebin space of all Riemannian metrics on $(M, g_0, \xi_0,
\eta_0)$ Sasakian of dimension $2n+1$.

We define an immersion  $\Gamma:  \H_S \rightarrow \M$ that maps $\phi \mapsto
g_\phi$ as defined in \eqref{defoII_tensors}.
As in the K\"ahler case, it is injective.
Indeed if two basic function $\phi_1, \phi_2 \in \H_S$ give rise to the same
Sasakian metric, taking the corresponding transverse structures we would have
$dd^c(\phi_1 - \phi_2) = 0$ forcing $\phi_1 - \phi_2 = \const$. The
normalization $I(\cdot) = 0$ then implies $\phi_1 = \phi_2$.

Let us compute the differential of $\Gamma$.
Let $\phi(t)$ be a curve in $\H_S$ with $\phi(0) = \phi$ and $\phi'(0) = \psi \in
T_\phi \H_S$. Then
\begin{equation} \label{pushH}
\Gamma_{* \phi} \psi = \frac d {dt} \biggl |_{t=0} g_{\phi(t)} = dd^c \psi(\Phi_0
\otimes \id) + 2 d^c \psi \odot \eta_\phi
\end{equation}
with the convention $a \odot b = \frac 1 2 (a \otimes b + b \otimes a)$. For
easier notation we call $\beta_\psi:= dd^c \psi(\Phi_0 \otimes 1) $.

The differential of $\Gamma$ is also injective. Indeed if $\psi$ is in its
kernel, then
\[
0 = \Gamma_{* \phi} \psi (\xi, \cdot) = d^c \psi,
\]
forcing $\psi$ to be zero, as it has zero integral.

On $T_{g} \M$ recall that the Ebin metric is given by, for $a, b \in T_g \M = \Gamma(S^2 M)$,
\[
g_\textup{E}( a, b)_{g} = \int_M g(a, b) dv_{g}.
\]

We want to compute the restriction of the Ebin metric on the space $ \H_S$.
\begin{prop}
The restriction of the Ebin metric to $ \H_S$ is twice the sum of the Calabi
metric with the gradient metric
\[
\frac 1 2 \Gamma^* g_\textup{E} = g_\textup{C} +  g_\textup{G}
\]
which we have called the \emph{sum metric}.
\end{prop}
\begin{proof}
Computing the length with respect to $g_\phi$ of the tensor in \eqref{pushH} we
get
\begin{align*}
| \beta_\psi + 2 d^c \psi \odot \eta_\phi|^2_{g_\phi} &= g_\phi(\beta_\psi,
\beta_\psi) +  2 g_\phi ( d^c \psi \otimes \eta_\phi, d^c \psi \otimes
\eta_\phi) + 2 g_\phi(\beta_\psi, 2 d^c \psi \odot \eta_\phi) \\
									&=
g_\phi (\beta_\psi, \beta_\psi) + 2 g_\phi(d^c \psi, d^c \psi) g_\phi(\eta_\phi,
\eta_\phi)  +  2\beta_\psi( (d^c \psi)^\sharp, \xi) \\
									&=
g_\phi (\beta_\psi, \beta_\psi) + 2 g_\phi(d^c \psi, d^c \psi)
\end{align*}
using the fact that the $g_\phi$-dual of $\eta_\phi$ is $\xi$, that the $\sharp$
is done with respect to $g_\phi$ and finally the fact that the tensor
$\beta_\psi$ is transverse, i.e. vanishes when evaluated on $\xi$.

Integrating with respect to $d\mu_\phi$ we have
\[
\langle \Gamma_{*\phi} \psi, \Gamma_{*\phi} \psi \rangle_\phi = \| \beta_\psi
\|^2_\phi + 2 \| d^c \psi \|^2_\phi
\]
where the right hand side are $L^2$ norms with respect to the metric $g_\phi$.
The second summand is twice the gradient metric on $ \H_S$ given by
\[
g_\textup{G}(\psi, \psi) = \int_M g_\phi(d \psi, d\psi) \Svol{\phi}{n}.
\]
(For a basic function, there is no difference between its Riemannian gradient and its basic gradient).

We now want to establish a useful formula that we will need in a while. Fix
$\phi \in \H_S$ and $h \in T_\phi \H_S$ we consider the curve $\phi(t) = \phi + t h$
which is in $\H_S$ for small $t$. We then compute for every curve $f(t) \in T_\phi
\H_S$,
\[
0 = \frac{d}{dt} \biggr |_{t=0} \int_M \Delta_{\phi(t)} f \Svol{\phi(t)}{n}
= \int_M (
\Delta_{\phi} f'(t) - (dd^c f, dd^c h)_{\phi} + \Delta_\phi f \Delta_\phi h)
\Svol{\phi}{n}.
\]
which means that $$g_\textup{C}(f, h)_\phi = \int_M (dd^cf, dd^c h)_\phi \Svol{\phi}{n}.$$

Then we have, since $\beta_\psi$ is the (transverse) $2$-tensor associated to
the basic form $dd^c \psi$, whose point-wise norms are related by $|\beta_\psi|^2
= 2 |dd^c \psi|^2$,
\[
g_\textup{C}(\psi, \psi) = \int_M (\Delta_\phi \psi)^2 \Svol{\phi}{n} = \int_M (dd^c
\psi, dd^c \psi)_{\omega_\phi} \Svol{\phi}{n} = \frac 1 2 \| \beta_\psi
\|^2_\phi.
\]
\end{proof}

\subsection{The sum metric on $\H_S$} \label{sec:LC}
Consider on $ \H_S$ the metric $g=2 g_\textup{C} + 2 g_\textup{G}$. It can be written, for $\phi \in
\H_S$ and $\alpha, \beta \in T_\phi  \H_S$,
\begin{align*}
g(\alpha, \beta) 	&= 2 \int_M \Delta_\phi \alpha \Delta_\phi \beta
\Svol{\phi}{n} - 2 \int_M \alpha \Delta_\phi \beta \Svol{\phi}{n} \\
				&= 2 \int_M \Delta_\phi (\alpha - G_\phi \alpha)
\Delta_\phi \beta \Svol{\phi}{n} \\
				&= g_\textup{C} (L_\phi \alpha, \beta)
\end{align*}
where $L_\phi = 2 (I - G_\phi)$ with $G_\phi$ the Green operator associated to $\Delta_\phi$.

Note that  the $G_\phi$ acting on functions with zero integral with respect to
$d\mu_\phi$ is the inverse of $\Delta_\phi$, since the projection on the space
of harmonic functions is
\[
H_\phi: f \mapsto \frac 1 {\vol_{g_\phi}} \int_M f
\Svol{\phi}{n} = 0
\]
and because of the known relation $I = H_\phi +\Delta_\phi G_\phi$.

We have the first result.
\begin{prop} \label{thmLCsum}
For any curve $\phi$ in $ \H_S$ and any section $v$ on $\phi$, the only solution
$D_t v$ of
\begin{equation}
\frac 1 2 L_\phi D_t v = D_t^C  v - G_\phi D_t^G v
\end{equation}
is the Levi-Civita covariant derivative of $g$, i.e. it is torsion free and
\begin{equation} \label{metricD}
\frac d {dt} g(v, v) = 2 g(D_t v, v).
\end{equation}
\end{prop}

Its proof is analogous to Proposition \ref{prop:LCcomb} and makes use of the results in \cite{EKA} about transversally elliptic operators.
The geodesic equation is then
\begin{equation} \label{eqgeod}
\Delta_\phi ^2 D_t^C \phi' - \Delta_\phi D_t^G \phi' = 0
\end{equation}
which is rewritten as \eqref{equa:geodesicsforsummetric}, i.e.

\begin{equation} \label{equa:geodesicsforsummetricsasaki}
(\Delta_\phi - I) \biggl ( ( \Delta_\phi \phi' )' + \frac 1 2 (\Delta_\phi
\phi')^2 \biggr ) - \frac 1 2 | i \de \debar \phi' |^2_\phi=0.
\end{equation}

\begin{rem}
It is clear that a curve $\phi$ which is a geodesic for both the Calabi and the
gradient metric would be a geodesic for our metric as well. Unfortunately there
are no such nontrivial curves, as one can easily see from the equations.
\end{rem}
\subsection{Another space of Sasakian metrics, an open problem}
Back to Sasakian geometry, it is interesting to consider also the space $\call G$ of Sasakian structures that share the same underlying CR structure. These deformations are known as \emph{type I} and we refer to \cite[Chap.~8]{monoBG}.  The most striking differences between $\call G$ and the $\H_S$ is that the former is finite dimensional and the metrics in it do not have the same volume. Recently, it was studied by Boyer, Huang, Legendre and T\o nnesen-Friedman \cite{BHLT} in relation to the existence of constant scalar curvature Sasakian metrics.

It would be interesting to compute the restriction of the Ebin metric to  $\call G \subset \M$ and study its intrinsic and extrinsic geometry.
\providecommand{\bysame}{\leavevmode\hbox to3em{\hrulefill}\thinspace}
\providecommand{\MR}{\relax\ifhmode\unskip\space\fi MR }
\providecommand{\MRhref}[2]{%
  \href{http://www.ams.org/mathscinet-getitem?mr=#1}{#2}
}
\providecommand{\href}[2]{#2}

\end{document}